 \documentclass[12pt,a4paper]{amsart}

 \usepackage{amsmath}
 \usepackage{amsfonts}
\usepackage{amssymb}

\newtheorem{Th}{Theorem}[section]
\newtheorem{Lemma}[Th]{Lemma}
\newtheorem{Cor}[Th]{Corollary}

 \theoremstyle{definition}
\newtheorem{Def}[Th]{Definition}

\newtheorem{Rem}[Th]{Remark}
\newtheorem{?}[Th]{Problem}

\newcommand{\ZZ}{\ensuremath{\mathbb{Z}}}

\newcommand{\RR}{\ensuremath{\mathbb{R}}}

\def \e {\varepsilon}

\begin{document}

\title[Difference sets and exponential sums II: cubic residues]{Difference sets and positive exponential sums II: cubic residues in cyclic groups}

\author[M. Matolcsi]{M\'at\'e Matolcsi}
\address{M\'at\'e Matolcsi: Budapest University of Technology and Economics (BME),
H-1111, Egry J. u. 1, Budapest, Hungary, and Alfr\'ed R\'enyi Institute of Mathematics, POB 127 H-1364 Budapest, Hungary.}
\email{matomate@renyi.hu}

\author[I. Z. Ruzsa]{Imre Z. Ruzsa}
\address{I. Z. R.: Alfr\'ed R\'enyi Institute of Mathematics,
Hungarian Academy of Sciences POB 127 H-1364 Budapest, Hungary}
\email{ruzsa@renyi.hu}

\begin{abstract}
By constructing suitable nonnegative exponential sums we give upper bounds on the cardinality of any set $B_q$ in cyclic groups $\ZZ_q$ such that the difference set $B_q-B_q$ avoids cubic residues modulo $q$.
\end{abstract}

\thanks{M. Matolcsi was supported by the NKFIH grant no. K132097 and K129335, I Z. Ruzsa was supported by NKFIH grant no. K129335.}

\maketitle

\bigskip

\section{Introduction}

\medskip

This paper is a follow-up to \cite{int1}, where we described the general properties connecting positive exponential sums and difference sets. In this second part, we apply the techniques of \cite{int1} to investigate the intersective properties of cubic residues in cyclic groups $\ZZ_q$. That is, by constructing suitable nonnegative exponential sums we obtain upper bounds on the cardinality of any set $B_q \subset \ZZ_q$ such that the difference set $B_q-B_q$ avoids cubic residues modulo $q$. We will turn to the case of quadratic residues in a later publication, but let us mention here that Gabdullin \cite{gab} has recently achieved a non-trivial upper bound by a different method.

\medskip

The ultimate aim of this research is to extend these results to the case of integers, i.e. give a strong upper bound on the cardinality of a set $B\subset \{0, 1, \dots, N\}$ such that $B-B$ does not contain any cubes or, more generally, $k$th powers for some fixed $k$.

\medskip

In order to see the relation of the modular case to the integer case we introduce the following notions.

\medskip

We use the standard notation $e(x)=e^{2\pi i x}$ throughout the paper.

\begin{Def}
A nonnegative function
\begin{equation}\label{witness}
f^{(n)}(x)=a_0+ \sum_{j=1}^{n-1} a_j (e(j^k x)+e(-j^k x)) \ge 0, \ x\in [0,1], \ \ f^{(n)}(0)=1,
\end{equation}
is called a $k^{th}$ {\it power witness function} of order $n$.  Similarly, a nonnegative function
\begin{equation}\label{witnessq}
g^{(q)}(y)=b_0+ \sum_{j=1, q\nmid j^k}^{q-1} b_j (e(j^k y/q)+e(-j^k y/q)) \ge 0, \ y=0,1,\dots , q-1, \ g^{(q)}(0)=1,
\end{equation}
is called a $k^{th}$ {\it power modular witness function}, modulo $q$. (The dependence of $a_j$ on $n$, and $b_j$ on $q$ will usually be suppressed in the notation.)
\end{Def}

\medskip

The normalization $f^{(n)}(0)=g^{(q)}(0)=1$ is only for convenience, and will be used throughout this note. We remark that the coefficients $a_j$ are automatically real as they are the Fourier (cosine) coefficients of the real valued function $f^{(n)}$. Also, we can assume without loss of generality (after averaging over $j$ if necessary) that $b_j=b_{j'}$ whenever $j^k\equiv \pm j'^k$ $mod \ q$. This makes the coefficients $b_j$ real (because $g^{(q)}$ is a real valued, even function, and $b_j$ are basically its Fourier coefficients, so they must be conjugate symmetric $b_j=\overline{b}_{-j}$, and they are also symmetric $b_j=b_{-j}$ by the assumption above, so they are in fact real).

\medskip

Let us clarify here the connection of these definitions to the results of \cite{int1}. Consider the cyclic group $G=\ZZ_q$ of order $q$, and let $R_0^{(k)}$ denote the set of $k$th power residues $mod \ q$ (including 0). Let $R_s^{(k)}$ denote the symmetrized set $R_s^{(k)}=R_0^{(k)}\cup -R_0^{(k)}$. We can think of the coefficients $b_j$ as a function $h: G\to \RR$, supported on $R_s^{(k)}$, defined in the following way: for any $r\in R_s^{(k)}$ let $h(r)=\sum_{j^3\equiv \pm r} b_j$. The Fourier transform of $h$ is then $\hat{h}=g^{(q)}\ge 0$. Therefore, the function $h$ fits into the class of functions used in the definition of $\lambda(R_s^{(k)})$ in \cite{int1}.

\medskip

We will restrict our attention to the case $k=3$ in this paper, the treatment of other odd values of $k$ being similar. We will use the following notations in accordance with the ones introduced in \cite{int1}.

\begin{Def}\label{ldef}
Let  $C^{(q)}$ denote the set of non-zero  cubic residues modulo $q$, and let  $C_0^{(q)}=C^{(q)}\cup \{0\}$. Recall from \cite{int1} that  $\delta(C_0^{(q)})$ denotes the maximal density (i.e. $|B_q|/q$) of a set $B_q\subset \ZZ_q$ such that $(B_q-B_q)\cap C^{(q)}=\emptyset$. Also, $\lambda(C_0^{(q)})$ denotes the minimal possible value of the constant term $b_0$ in expression \eqref{witnessq} for $k=3$.
\end{Def}

\medskip

In this note we will be interested in upper bounds on $\delta(C_0^{(q)})$, but let us also mention here what is known about lower bounds. For a prime $q=3k+1$, it is clear by the Ramsey number estimate $R(s,s)\le 4^s$ that one can find a set $B\subset \ZZ_q$ of size at least $\lfloor \log_4 q \rfloor$ such that $(B_q-B_q)\cap C^{(q)}=\emptyset$. (After applying the Ramsey estimate, one can either find a suitable set $B$, or a set $B'\subset \ZZ_q$, $|B'|\ge \lfloor \log_4 q \rfloor$ such that $B'-B'\subset C_0^{(q)}$. In the latter case, it suffices to consider $B=tB'$ for any non-cubic residue $t\in \ZZ_q.$) Also, a construction of a set $B'$ such that $B'-B'\subset C_0^{(q)}$, $\log_6 q\lessapprox |B|$, not using Ramsey theory, appears in Theorem 3 of \cite{cohen}. For a prime $q=3k+2$ the problem is trivial as $|B|\le 1$. For a square-free integer $q=s_1s_2\dots s_m$, where $s_i=3k_i+1$, a straightforward direct product construction (see Section 8 in \cite{int1}) gives a suitable set $B$ with $|B|\ge \prod_{i=1}^m \lfloor \log_4 s_i \rfloor$. We are not aware of any stronger results, and numerical experiments seem to indicate that the logarithmic size is not far from the truth for primes $q=3k+1$.

\medskip

We will obtain upper bounds on $\lambda(C_0^{(q)})$, and these will automatically imply upper bounds on $\delta(C_0^{(q)})$. In particular, according to Theorem 1.4 in \cite{int1}, any modular witness function $g^{(q)}$ in equation \eqref{witnessq} testifies an upper bound $|B_q|\le b_0 q$ for the cardinality of a set $B_q\subset \ZZ_q$ such that $B_q-B_q$ avoids the $k^{th}$-power residues modulo $q$. Similarly, a $k^{th}$ power witness function $f^{(n)}$ will lead to an estimate $|B_0|\lessapprox a_0 n^k$ for the cardinality of a set $B_0\subset [1,\dots n^k]$ such that $B_0-B_0$ avoids the $k^{th}$ powers. Our aim, therefore, is to minimize the values of $a_0$ and $b_0$ in equations \eqref{witness} and \eqref{witnessq}. We will only be concerned with the modular case in this note.

\medskip

Notice that $f^{(n)}(x)$ automatically induces a modular witness function modulo $q=n$ by defining $g^{(n)}(y)=f^{(n)}(y/n)$, and if $n$ is squarefree then $f^{(n)}$ and $g^{(n)}$ have the same constant terms $a_0=b_0$. Therefore, constructing modular $k^{th}$ power witness functions $g^{(q)}$ is {\it formally easier} than constructing $k^{th}$ power witness functions $f^{(n)}$. Conversely, in later publications of this series of papers we hope to construct a family of witness functions $f^{(n)}$ assuming that we already have a family of {\it self-compatible} modular witness functions $g^{(q)}$ at our disposal. Self-compatibility is a natural property defined here.

\begin{Def}
A family of $k^{th}$ power modular witness functions $g^{(q)}$ ($q=1,2, \dots$) of the form \eqref{witnessq} is called {\it self-compatible} if $g^{(q_1)}(y_1)=g^{(q_2)}(y_2)$ whenever $y_1/q_1=y_2/q_2$.
\end{Def}

This paper will be devoted entirely to constructing a self-compatible family of modular witness functions for $k=3$, which is a very interesting problem in itself.

\section{Cubic residues}\label{sec:cubic}

Somewhat surprisingly, the case of cubic residues is considerably easier to handle than that of quadratic residues. This is due to the fact that $k=3$ is an odd number and therefore -1 is automatically a cubic residue modulo $q$ for any $q$. This implies that the set of cubic residues modulo $q=p_1^{\alpha_1}\dots p_r^{\alpha_r}$ is always {\it symmetric} to $0$. Also, it is equal to the direct product of the sets of cubic residues modulo $p_1^{\alpha_1}, \ \dots, \ p_r^{\alpha_r}$. Therefore the results on direct products in \cite[Section 8]{int1} can be invoked, and the problem reduces to forming nonnegative exponential sums with cubic residue frequencies in cyclic groups $\ZZ_{p^\alpha}$ of prime-power order. The self-compatibility property will be an automatic consequence of the construction.

\medskip

We first consider the case of squarefree moduli $q$, which is of independent interest in itself.

\medskip

\begin{Th}\label{sfc}
Let $q=p_1p_2\dots p_rs_1s_2\dots s_m$ be a squarefree integer, where $p_i$ denote primes of the form $3k+2$, $s_l$ denote primes of the form $3k+1$, and if $3|q$ then we list the prime 3 among the $p_i$. There exists a cubic witness function modulo $q$, $g^{(q)}(y)=b_0+ \sum_{j=1}^{q-1} b_j (e(j^3 y/q)+e(-j^3 y/q)) \ge 0$, $y=0,1,\dots , q-1$, $g^{(q)}(0)=1$, such that
\begin{equation}\label{sfcbound}
b_0\le \left ( \prod_{i=1}^{r}\frac{1}{p_i} \right ) \left ( \prod_{l=1}^{m}\frac{2}{\sqrt{s_l}} \right ) \le \frac{2^m}{\sqrt{q}}.
\end{equation}
That is, with the notation of Definition \ref{ldef}, we have
\begin{equation}\label{csflambda}
\lambda(C_0^{(q)})\le \frac{2^m}{\sqrt{q}}=O_\varepsilon(q^{-\frac{1}{2}+\varepsilon})
\end{equation}
for every $\varepsilon >0$.
\end{Th}

\begin{proof}
For a prime $p=3k+2$, all elements of $\ZZ_{p}$ are cubic residues. Therefore, we can take the trivial witness function modulo $p$,
\begin{equation}\label{3kp2}
g^{(p)}(y)=\frac{1}{p}\sum_{j=0}^{p-1} e(j^3y/p),
\end{equation}
with the constant term being $\frac{1}{p}$. That is, $\lambda(C_0^{(p)})\le \frac{1}{p}$ (in fact, equality holds). The same is true for $p=3$, which is listed among the $p_i$ if $3|q$.

\medskip

For a prime $s=3k+1$ the set $C^{(s)}\subset \ZZ_{s}$ of non-zero cubic residues modulo $s$ is symmetric to zero, and consists of $\frac{s-1}{3}$ elements. Denoting the cubic multiplicative characters on $\ZZ_s$ by $\chi_0, \chi_1, \chi_2$ (with $\chi_0$ being the principal character), we have
\begin{equation}
\sum_{j=1}^{s-1}e(j^3y/s)=\sum_{l=1}^{s-1}e(ly/s)\chi_0(l)+\sum_{l=1}^{s-1}e(ly/s)\chi_1(l)+\sum_{l=1}^{s-1}e(ly/s)\chi_2(l)\ge -1-2\sqrt{s},
\end{equation}
because the last two sums have absolute value $\sqrt{s}$ (see \cite[Lemma 4.3]{vaughan}). Therefore, after normalization, we may take
\begin{equation}\label{3kp1}
g^{(s)}(y)=\frac{2\sqrt{s}+1}{2\sqrt{s}+s}+\frac{1}{2\sqrt{s}+s}\sum_{j=1}^{s-1}e(j^3y/s)\ge 0, \ (y=0, \dots, s-1),\ g^{(s)}(0)=1,
\end{equation}
and the constant term satisfies $\frac{2\sqrt{s}+1}{2\sqrt{s}+s} \le \frac{2}{\sqrt{s}}$. That is, $\lambda(C_0^{(s)})\le \frac{2}{\sqrt{s}}$

\medskip

Finally, the group $\ZZ_q$ is the direct product of the groups $\ZZ_{p_i}$ and $\ZZ_{s_l}$, and the set of cubic residues $C_0^{(q)}$ modulo $q$ is symmetric to zero and equals the direct product of the sets $C_0^{(p_i)}$ and $C_0^{(s_l)}$. Therefore, the direct product construction of \cite[Theorem 8.1]{int1} can be applied, and equation (35) in \cite{int1} implies $\lambda(C_0^{(q)})=\left ( \prod_{i=1}^{r}\lambda(C_0^{(p_i)}) \right ) \left ( \prod_{l=1}^{m}\lambda(C_0^{(s_l)}) \right )$, and the bound given in \eqref{sfcbound} follows. Finally, estimate \eqref{csflambda} follows from the fact that $2^m=O(q^\varepsilon)$. \end{proof}

\medskip

\begin{Rem}
For primes of the form $s=3k+1$ the estimate $\lambda(C_0^{(s)})\le \frac{2}{\sqrt{s}}$ is asymptotically optimal, and therefore the exponent in the bound $\lambda(C_0^{(q)})=O(q^{-\frac{1}{2}+\varepsilon})$ cannot be improved.
\end{Rem}

\medskip

Via Theorem 1.4 in \cite{int1} the bound \eqref{csflambda} above immediately implies the same upper bound for the density of sets avoiding cubic residues:

\begin{Cor}
For any $\varepsilon >0$ and any square-free positive integer $q$, the density $\delta(B_q)$ of any set $B_q \subset \ZZ_q$ such that the difference set $B_q-B_q$ avoids cubic residues modulo $q$ satisfies $\delta(B_q)=O_\varepsilon(q^{-\frac{1}{2}+\varepsilon})$.
\end{Cor}

\medskip

Now we turn to general (non-squarefree) moduli $q$. The direct product construction can still be applied in this case, but the construction of witness functions modulo prime-powers is slightly more technical.

\medskip

\begin{Lemma}\label{primepower}
Using the notation of Definition \ref{ldef}, we have
\begin{equation}\label{lambdapn3}
\lambda(C_0^{(p^m)})=\lambda(C_0^{(p)})^{[(m+2)/3]}
\end{equation}
for any prime $p\ne 3$ and any integer $m\ge 1$ (the notation $[x]$ stands for the integer part of any number $x$). For $p=3$ we have
\begin{equation}\label{lambda3}
\lambda(C_0^{(3^{3m})})=\lambda(C_0^{(27)})^{m}.
\end{equation}
\end{Lemma}
\begin{proof}
Let $p\ne 3$. We prove the lemma by induction on $m$. For $m=1$ the statement is trivial. Let $m\ge 2$, and let us identify the residue classes modulo $p^m$ with the numbers $0, 1, \dots, p^m-1$.

\medskip

The structure of the cubic residues modulo $p^m$ is the following: for $0<t\le p^{m-1}-1$ and $0\le l\le p-1$ we have $t+lp^{m-1}\in C^{p^m}$ if and only if $t\in C^{p^{m-1}}$. For $t=0$, we have two different cases. If 3 does not divide $m-1$ then $lp^{m-1}\in C^{p^m}$ if and only if $l=0$. If $3|m-1$ then $lp^{m-1}\in C^{p^m}$ if and only if $l\in C_0^{(p)}$.

\medskip

Consider the subgroup $H=\{0, p^{m-1}, \dots, (p-1)(p^{m-1})\}\subset \ZZ_{p^m}$. Then $H\equiv \ZZ_p$, $\ZZ_{p^m}/H\equiv \ZZ_{p^{m-1}}$ and we can apply the general result concerning subgroups and factor groups in Theorem 7.2 in \cite{int1}. Namely, by formula (29) in \cite{int1}, and the structure of $C_0^{(p^m)}$ described above, we obtain $\lambda (C_0^{(p^m)}) \ge \lambda (C_0^{(p^{m-1})})$ if 3 does not divide $m-1$, and $\lambda (C_0^{(p^m)}) \ge \lambda (C_0^{(p^{m-1})}) \lambda (C_0^{(p)})$ if $3|m-1$. This proves the lower bound in the inductive step.

\medskip

To prove the upper bound, we apply a simple construction. Let
\begin{equation}
g^{(p)}(y)=b_0^{(p)}+\sum_{j\in C^{(p)}} b_j^{(p)}e(jy/p)
\end{equation}
be a witness function modulo $p$ such that $b_0^{(p)}$ is minimal, i.e. $b_0^{(p)}=\lambda(C_0^{(p)})$ (a simple compactness argument shows that the minimal value of $b_0^{(p)}$ can indeed be attained). Similarly, let
\begin{equation}
g^{(p^{m-1})}(y)=b_0^{(p^{m-1})}+\sum_{t\in C^{(p^{m-1})}} b_t^{(p^{m-1})}e(ty/p^{m-1})
\end{equation}
be a witness function modulo $p^{m-1}$ such that $b_0^{(p^{m-1})}=\lambda(C_0^{(p^{m-1})})$. Define
\begin{equation}
g^{(p^{m})}(z)=\sum_{t\in C^{(p^{m-1})}} \frac{b_t^{(p^{m-1})}}{p}\sum_{l=0}^{p-1}e((t+lp^{m-1})z/p^{m})+ h_0(z)
\end{equation}
where $h_0(z)=b_0^{(p^{m-1})}$ if 3 does not divide $m-1$, and $h_0(z)=b_0^{(p^{m-1})} (b_0^{(p)}+\sum_{j\in C^{(p)}} b_j^{(p)}e(jz/p))$ if $3|m-1$. It is clear that $g^{(p^{m})}(0)=1$ and the constant term of $g^{(p^{m})}(z)$ is equal to $b_0^{(p^{m-1})}=\lambda(C_0^{(p^{m-1})})$ if 3 does not divide $m-1$, while it is $b_0^{(p^{m-1})}b_0^{(p)}=\lambda(C_0^{(p^{m-1})})\lambda(C_0^{(p)})$ if $3|m-1$. We claim that $g^{(p^{m})}(z)$ is indeed a witness function modulo $p^m$, i.e. $g^{(p^{m})}(z)\ge 0$ for $z=0,1,\dots p^m-1$. If $z=py$ then a substitution into the definitions shows that
\begin{equation}\label{selfcomp}
g^{(p^{m})}(z)=g^{(p^{m})}(py)=g^{(p^{m-1})}(y)\ge 0.
\end{equation}
If $z=py+j$ where $j\ne 0$, then for each $t\in C^{(p^{m-1})}$ we have $\sum_{l=0}^{p-1}e((t+lp^{m-1})z/p^{m})=e(t(py+j)/p^m)\sum_{l=0}^{p-1}e(lj/p)=0$, therefore
\begin{equation}
g^{(p^{m})}(z)=g^{(p^{m})}(py+j)=h_0(py+j)=b_0^{(p^{m-1})}g^{(p)}(j)\ge 0.
\end{equation}
This proves the upper bound in the inductive step and completes the proof of the Lemma for $p\ne 3$.

\medskip

For $p=3$ the proof is similar. The structure of the cubic residues modulo $3^{3m}$ is the following: for $0<t\le 3^{3(m-1)}-1$ and $0\le l\le 26$ we have $t+l3^{3(m-1)}\in C^{3^{3m}}$ if and only if $t\in C^{3^{3(m-1)}}$. For $t=0$, we have $l3^{3(m-1)}\in C^{3^{3m}}$ if and only if $l\in C_0^{(27)}$.

\medskip

Consider the subgroup $H=\{0, 3^{3(m-1)}, \dots, 26(3^{3(m-1)})\}\subset \ZZ_{3^{3m}}$. Then $H\equiv \ZZ_{27}$, $\ZZ_{3^{3m}}/H\equiv \ZZ_{3^{3(m-1)}}$, and formula (29) in \cite{int1} implies $\lambda (C_0^{(3^{3m})}) \ge \lambda (C_0^{(3^{3(m-1)})}) \lambda (C_0^{(27)})$. This proves the lower bound in the inductive step.

\medskip

To prove the upper bound we define
\begin{equation}
g^{(3^{3m})}(z)=\sum_{t\in C^{(3^{3(m-1)})}} \frac{b_t^{(3^{3(m-1)})}}{27}\sum_{l=0}^{26}e((t+l3^{3(m-1)})z/3^{3m})+b_0^{(3^{3(m-1)})} (b_0^{(27)}+\sum_{j\in C^{(27)}} b_j^{(27)}e(jz/27)),
\end{equation}
and note that $g^{(3^{3m})}(0)=1$, the constant term is $\lambda(C_0^{(3^{3(m-1)})})\lambda(C_0^{(27)})$, and for $z=27y$ we have
\begin{equation}\label{selfcomp3}
g^{(3^{3m})}(z)=g^{(3^{3m})}(27y)=g^{(3^{3(m-1)})}(y)\ge 0,
\end{equation}
and for $z=27y+j$ where $j\ne 0$ we have
\begin{equation}
g^{(3^{3m})}(z)=g^{(3^{3m})}(27y+j)=b_0^{(3^{3(m-1)})}g^{(27)}(j)\ge 0.
\end{equation}
\end{proof}

This lemma enables us to prove our main result concerning cubic residues in cyclic groups.

\begin{Th}\label{main3}
For $\e=-\frac{\log \left (1-\frac{2}{2+\cos (\pi/13)+\sin (3\pi/26)}\right )}{3\log 13}\approx 0.1195$ there exists a self-compatible family $g^{(q)}(y)$ of cubic modular witness functions of the form \eqref{witnessq} such that $b_0\le q^{-\e}$. In particular,
\begin{equation}\label{cgeneral}
\lambda(C_0^{(q)})\le q^{-\e}
\end{equation}
for every $q\ge 1$.
\end{Th}
\begin{proof}
If $q\ne 3$ is a prime of the form $3k+2$ we define the modular witness function $g^{(q)}(y)$ by equation \eqref{3kp2}, and note that $\lambda(C_0^{(q)})=q^{-1}$. If $q$ is a prime of the form $3k+1$ we define $g^{(q)}(y)=b_0^{(q)}+b^{(q)}\sum_{j=1}^{q-1}e(j^3y/q)$ such that $b_0^{(q)}+(q-1)b^{(q)}=1$, $g^{(q)}(y)\ge 0$ for $0\le y\le q-1$, and $b_0^{(q)}$ is minimal possible,  i.e. $b_0^{(q)}=\lambda(C_0^{(q)})$ (it is easy to see that all the coefficients of $e(j^3y/q)$ can be assumed to be equal by averaging; cf. \cite[Proposition 5.2]{int1}). By equation \eqref{3kp1} we note that $\lambda(C_0^{(q)})\le \frac{2\sqrt{q}+1}{2\sqrt{q}+q}\le q^{-0.36}$ for $q\ge 31$. For $q=7, 13$, and $19$ direct computation shows that $\lambda(C_0^{(q)})\le q^{-3\e}$ with equality for $q=13$.

\medskip

If $q=p^m$ is a power of a prime $p\ne 3$ then we define $g^{(q)}(y)$ by induction on $m$ as in Lemma \ref{primepower}. Equation \eqref{lambdapn3} shows that $\lambda(C_0^{(q)})\le q^{-\e}$. The self-compatibility property follows from equation \eqref{selfcomp}.

\medskip

For $q=27$ we define $g^{(27)}(y)=b_0^{(27)}+ b^{(27)}\sum_{j\in C^{(27)}}e(jy/27)$ such that $g^{(27)}(0)=1$, $g^{(27)}(y)\ge 0$ for $y=0,1, \dots, 26$, and $b_0^{(27)}=\lambda(C_0^{(27)})$ (again, the non-leading coefficients can be assumed to be equal by averaging). For $q=3^{3m}$ we define $g^{(q)}(y)$ by induction on $m$ as in Lemma \ref{primepower}. For $q=3^{3m-1}$ and $q=3^{3m-2}$ we define $g^{(q)}(y)=g^{(3^m)}(3y)$ and $g^{(q)}(y)=g^{(3^m)}(9y)$, respectively. Self-compatibility follows from this definition and equation \eqref{selfcomp3}. Direct computation of $\lambda(C_0^{(27)})$ and equation \eqref{lambda3} shows that for any $q=3^\alpha$ we have $\lambda(C_0^{(q)})\le q^{-\e}$.

\medskip

Finally, let $q=p_0^{\alpha_0}p_1^{\alpha_1}\dots p_r^{\alpha_r}$ be the prime factorization of $q$, where $p_0=3$ if it appears in $q$. The set of cubic residues $C_0^{(q)}$ modulo $q$ is symmetric to zero and equals the direct product of the sets $C_0^{(p^\alpha)}$. Furthermore, as in the construction of
\cite[Theorem 8.1]{int1} we define the cubic witness function $g^{(q)}(y)=b_0^{(q)}+\sum_{j\in C^{(q)}}b_j^{(q)}e(jy/q)$ as the direct product of the cubic modular witness functions $g^{(p^\alpha)}$. It is straightforward that the self-compatibility property is preserved under direct products, and $\lambda(C_0^{(q)})\le b_0^{(q)}=b_0^{(3^{\alpha_0})} \left ( \prod_{i=1}^{r}\lambda(C_0^{(p_i^{\alpha_r})}) \right )\le q^{-\e}$.
\end{proof}

Again, via Theorem 1.4 in \cite{int1} the bound \eqref{cgeneral} above immediately implies the same upper bound for the density of sets avoiding cubic residues:

\begin{Cor}
For $\e=-\frac{\log \left (1-\frac{2}{2+\cos (\pi/13)+\sin (3\pi/26)}\right )}{3\log 13}\approx 0.1195$ and any positive integer $q$, the density $\delta(B_q)$ of any set $B_q \subset \ZZ_q$ such that the difference set $B_q-B_q$ avoids cubic residues modulo $q$ satisfies $\delta(B_q)\le q^{-\e}$.
\end{Cor}

\centerline{\textsc{Acknowledgement}}

The authors are grateful for the reviewer for his/her suggestions which have improved the presentation of the paper.

\end{document}